\documentclass[11pt]{article}

\usepackage{amsmath,amsfonts,amssymb,amsthm}

\title{Kochen-Specker sets and Hadamard matrices}
\author{
Petr Lison\v{e}k\thanks{Research was supported
in part by the Natural Sciences and Engineering Research Council of Canada
(NSERC).}
\\
Department of Mathematics\\
Simon Fraser University\\
Burnaby, BC, V5A 1S6\\
Canada\\
\  \\
{\tt plisonek@sfu.ca}
}

\newtheorem{theorem}{Theorem}[section]

\newtheorem{proposition}[theorem]{Proposition}

\theoremstyle{definition}
\newtheorem{definition}[theorem]{Definition}
\newtheorem{example}[theorem]{Example}

\def\la{\langle}
\def\ra{\rangle}

\def\Z{{\mathbb Z}}
\def\C{{\mathbb C}}

\def\V{{\mathcal V}}
\def\B{{\mathcal B}}

\def\z{\zeta}

\begin{document}

\maketitle

\begin{abstract}
We introduce a new class of complex Hadamard matrices
which have not been studied previously. We use these matrices
to construct a new infinite family of parity proofs
of the Kochen-Specker theorem. 
We show that the recently discovered simple parity proof
of the Kochen-Specker theorem is the initial member of this infinite family.
\end{abstract}

\section{Introduction}
Kochen-Specker theorem is an important result in quantum mechanics 
\cite{KS-paper}.
It demonstrates the contextuality of quantum mechanics,
which is one of its properties that may become crucial
in quantum information theory \cite{How}.
In this paper we focus 
on proofs of Kochen-Specker theorem that are given by showing
that, for $n\ge 3$, 
there does not exist a function $f:\C^n\rightarrow \{0,1\}$
such that for every orthogonal basis $B$ of $\C^n$ there
exists {\em exactly one} vector $x\in B$ such that $f(x)=1$
(where $\C^n$ denotes the $n$-dimensional vector space over the field
of complex numbers).
This particular approach has been used in many publications,
see for example \cite{Adan-18,Lis-PRA,WA,WA-preprint} 
and many references cited therein.
The following definition formalizes 
one common way of constructing such proofs.

\begin{definition}
\label{def-KS-pair}
We say that $(\V,\B)$ is a {\em Kochen-Specker pair in $\C^n$}
if it meets the following conditions:
\begin{itemize}
\item[(1)] 
$\V$ is a finite set of vectors in $\C^n$.
\item[(2)]
$\B=(B_1,\ldots,B_k)$ where $k$ is odd,
and for all $i=1,\ldots,k$
we have that $B_i$ is an orthogonal basis of $\C^n$
and $B_i\subset \V$.
\item[(3)] 
For each $v\in\V$ the number of $i$ such that $v\in B_i$ is even.
\end{itemize}
\end{definition}

Let us show that the existence of a Kochen-Specker pair 
%(henceforth
%abbreviated KS pair)
%satisfying Definition~\ref{def-KS-pair}
demonstrates
the non-existence of a function $f$ with the properties given above.
Towards a contradiction suppose that  $(\V,\B)$ satisfies 
Definition~\ref{def-KS-pair}
and $f:\C^n\rightarrow \{0,1\}$ has the properties specified above.
Denote $V_1=\{ x\in\V : f(x)=1\}$.
By conditions (2) and (3) and by properties of $f$,
the number of $i$ such that $|B_i\cap V_1|=1$
is even. Since the length of the list $\B$ is odd, 
there exists an $i$ such that $|B_i\cap V_1|\neq 1$, in contradiction
to the required properties of $f$. Since this contradiction is based
on a parity argument, the Kochen-Specker pairs introduced
in Definition~\ref{def-KS-pair}
are often called ``parity proofs of the Kochen-Specker theorem.''

It is quite common in the literature 
\cite{T7-experiment,Lis-PRA,WA-preprint}
to refer to a Kochen-Specker pair 
as {\em Kochen-Specker set,} and we will do so sometimes in this paper.
Kochen-Specker sets
are key tools for proving some fundamental results in quantum theory 
and they also have 
various potential applications in quantum information processing
\cite{T7-experiment}.  
In Section~\ref{sec-KS-construction}
we give a construction of
a family of Kochen-Specker sets in infinitely many different dimensions.
Before that, in Section~\ref{sec-SL-Had} we introduce a new class
of complex Hadamard matrices,
which are used in our construction, and
they may be also an interesting object of study on their own.
In Section~\ref{sec-conclusion} we draw some conclusions from our results.

\section{S-Hadamard matrices}
\label{sec-SL-Had}

For $z\in\C$ let $\overline z$ denote its complex conjugate.
We work with the usual inner product on $\C^n$
defined by $\la x,y\ra=\sum_{i=1}^n x_i\overline{y_i}$.
For a complex matrix $H$, let $H^*$ denote its conjugate transpose.
Let $I_n$ denote the $n\times n$ identity matrix.
We say that a complex number $z$ is {\em unimodular} if $|z|=1$.
We say that a vector $x\in\C^n$ is  {\em unimodular} if each coordinate
of $x$ is unimodular.

\begin{definition}
\label{def-SL-Had-m}
An $n\times n$ matrix $H=(h_{i,j})$ whose entries are complex numbers
is called {\em S-Hadamard matrix of order $n$} if it meets the following
conditions:
\begin{itemize}
\item[(1)]
$HH^*=nI_n$
\item[(2)]
$|h_{i,j}|=1$ for all $1\le i,j\le n$
\item[(3)]
for each $1\le k, \ell \le n$, $k\neq \ell$, we have 
$\sum_{j=1}^n h_{k,j}^2 \overline{h_{\ell,j}^2} =0$.
\end{itemize}
\end{definition}

Conditions (1) and (2) are the definition of Hadamard matrices
which are prominent objects 
in combinatorial design theory
that have been studied since the 19th century 
\cite{Horadam}, first as matrices over $\{-1,1\}$ and then
more generally as matrices over complex numbers.
Condition (3) appears to be a new condition not previously seen
in the literature. The proposed name {\em S-Hadamard matrix}
reflects the form of this new condition, which involves
{\bf \underline{s}}quares of the entries of~$H$.
All three conditions in Definition \ref{def-SL-Had-m}
are required for our construction
of Kochen-Specker pairs given in Theorem \ref{thm-Had-KS}.

Throughout this paper
we use additive notation for group operations.
By $\Z_g$ we denote the cyclic group of integers modulo~$g$.

\begin{definition}\cite[Definition 5.1]{handbook-GH}
\label{def-GH-mat}
Let $G$ be a group of order $g$ and let $\lambda$ be a positive integer.
A {\em generalized Hadamard matrix over $G$}
is a $g\lambda \times g\lambda$ matrix $M=(m_{i,j})$
whose entries are elements of $G$
and for each $1\le k<\ell\le g\lambda$, each element of $G$
occurs exactly $\lambda$ times among the differences $m_{k,j}-m_{\ell,j}$,
$1\le j\le g\lambda$.
Such matrix is denoted {\rm GH}$(g,\lambda)$.
\end{definition}

\begin{proposition}
\label{proposition-GH-SH}
Suppose that $g>2$ and {\rm GH}$(g,\lambda)$ over $\Z_g$ exists. 
Then there exists an S-Hadamard matrix of order $g\lambda$.
\end{proposition}
\begin{proof}
Assume that $M=(m_{i,j})$ is a GH$(g,\lambda)$ over $\Z_g$.
Let $\z_g=e^{2\pi\sqrt{-1}/g}$ be the primitive $g$-th root of unity
in $\C$. 
Define the 
$g\lambda \times g\lambda$ matrix $H=(h_{i,j})$ by $h_{i,j}=\z_g^{m_{i,j}}$
for all $i,j$.
We claim that $H$ is an S-Hadamard matrix of order $g\lambda$.
We will verify that the three conditions
in Definition~\ref{def-SL-Had-m} are satisfied.
Let $h_k$ and $h_{\ell}$ be two rows of $H$, $k\neq \ell$. Then
\[
\la h_k,h_{\ell}\ra=\sum_{j=1}^{g\lambda} \z_g^{m_{k,j}-m_{\ell,j}}
=\lambda\sum_{i=0}^{g-1}\z_g^i=0
\]
and $\la h_k,h_k\ra =\sum_{j=1}^{g\lambda} |h_{k,j}|^2=g\lambda$,
thus condition (1)
is satisfied. Condition (2) 
follows from the construction of~$H$.
To verify condition (3) we compute
\[
\sum_{j=1}^{g\lambda} h_{k,j}^2 \overline{h_{\ell,j}^2}
=\sum_{j=1}^{g\lambda} \z_g^{2(m_{k,j}-m_{\ell,j})}
=\lambda\sum_{i=0}^{g-1}\z_g^{2i}
=\lambda\frac{\z_g^{2g}-1}{\z_g^2-1}=0
\]
where we used the assumption that $g>2$, hence $\z_g^2\neq 1$.
\end{proof}

It would be interesting to find other constructions
of S-Hadamard matrices, and we pose this as an open problem.

\begin{example}
\label{example-GH-exist}
The construction
given in our main result (Theorem~\ref{thm-Had-KS}) requires
the existence an S-Hadamard matrix of {\em even} order.
For illustration let us consider small even orders.
Proposition \ref{proposition-GH-SH} allows one
to construct an S-Hadamard matrix of order $n$
for all even $0<n\le 100$ except
$n=2, 4, 8, 32, 40, 42, 60, 64, 66, 70, 78, 84, 88$.
The underlying generalized Hadamard matrices are obtained
by constructions
given in \cite{handbook-GH},
moreover for $n=16$ examples of GH(4,4) over $\Z_4$
are given in \cite{Harada}.
By consulting Table~6.1 in \cite{Lampio}
we see that no other suitable GH$(g,\lambda)$ over $\Z_g$ are known
for $g\le 7$ and $\lambda\le 10$.
For some of the values of $n$ which we excluded above,
the existence of a suitable generalized Hadamard matrix
is an open problem, and it is possible that 
S-Hadamard matrices of those orders may exist.
\end{example}

\section{An infinite family of Kochen-Specker sets}
\label{sec-KS-construction}

S-Hadamard matrices introduced in the previous section
will now be used to construct an infinite family of Kochen-Specker sets.

\begin{theorem}
\label{thm-Had-KS}
Suppose that there exists an S-Hadamard matrix of order~$n$
where $n$ is even.
Then there exists a Kochen-Specker pair $(\V,\B)$ in $\C^n$
such that $|\V| \le {{n+1}\choose 2}$ and $|\B|=n+1$.
\end{theorem}
\begin{proof}
First we construct the set $\V$. Let the elements of $\V$
be denoted $v^{\{r,s\}}$ where $1\le r,s\le n+1$, $r\neq s$.
Note that we use the standard convention that sets are unordered,
hence $v^{\{r,s\}}$ and $v^{\{s,r\}}$ denote the same element of $\V$,
for all $r\neq s$. 
For $x,y\in \C^n$
we define $x\circ y=(x_1y_1,\ldots,x_ny_n)$.

Let $H=(h_{i,j})$ be the S-Hadamard matrix of order~$n$
whose existence is assumed,
and let $h_i$ denote the $i$-th row of $H$. 
Without loss of generality we can assume that $h_1$
is the all-one vector, denoted $\bf 1$.
If this is not the case, then replace each entry $h_{ij}$ of $H$
with $h_{ij}h_{1j}^{-1}$; this operation preserves all
conditions of Definition~\ref{def-SL-Had-m}. 

We construct the elements of $\V$
as follows:
\begin{itemize}
\item
For $1 < s\le n+1$ let $v^{\{1,s\}}=h_{s-1}$.
\item
For $2 < s\le n+1$ let $v^{\{2,s\}}=h_{s-1}\circ h_{s-1}$.
\item
For $2 < r <s \le n+1$ let $v^{\{r,s\}}=h_{r-1}\circ h_{s-1}$.
\end{itemize}

For $1\le r\le n+1$ let $B_r=\{ v^{\{r,i\}}\;:\; 1\le i\le n+1,\; i\neq r\}$,
and let $\B=(B_1,\ldots,B_{n+1})$. We will now prove that each $B_r$ 
is an orthogonal basis of $\C^n$. Note that for $x,y,z\in\C^n$
such that $z$ is unimodular we have
\begin{equation}
\label{eq-circ-orth}
\la z\circ x, z\circ y\ra = 
\la x\circ z, y\circ z\ra = \sum_{i=1}^n x_iz_i\overline{y_iz_i}=\la x,y \ra.
\end{equation}
Since distinct rows of $H$ are orthogonal
and all rows of $H$ are unimodular, 
equation (\ref{eq-circ-orth}) proves 
\begin{equation*} 
\la v^{\{r,s\}},v^{\{r,t\}}\ra
=
\la h_{r-1}\circ h_{s-1}, h_{r-1}\circ h_{t-1}\ra
=
\la h_{s-1}, h_{t-1}\ra
=0
%\label{eq-orth}
\end{equation*} 
whenever 
\begin{equation} 
2< r,s,t\le n+1
\mbox{\ \ and\ \ } r,s,t \mbox{\ are\ distinct.}
\label{eq-easy-rst}
\end{equation} 

We will now prove the desired orthogonality relations
$\la v^{\{r,s\}},v^{\{r,t\}}\ra =0$
for those pairs
of vectors $v^{\{r,s\}},v^{\{r,t\}}$
which are not covered by condition (\ref{eq-easy-rst}).
We will split the proof into cases according to the value of~$r$.

Let $r=1$.
For $1< s < t \le n+1$ we have
\[
\la v^{\{1,s\}},v^{\{1,t\}}\ra 
=
\la h_{s-1},h_{t-1}\ra 
=
0.
\] 

Now let $r=2$. 
For $2< s < t \le n+1$ we have
\[
\la v^{\{2,s\}},v^{\{2,t\}}\ra 
=\la h_{s-1}\circ h_{s-1}, h_{t-1}\circ h_{t-1}\ra = 
\sum_{j=1}^n h_{s-1,j}^2 \overline{h_{t-1,j}^2} =0
\]
by condition (3) in Definition~\ref{def-SL-Had-m}.
For $2< t \le n+1$ we have
\[
\la v^{\{2,1\}},v^{\{2,t\}}\ra 
= \la {\bf 1} , h_{t-1}\circ h_{t-1} \ra 
= \sum_{j=1}^n \overline{h_{t-1,j}^2} =0
\]
by condition (3) in Definition~\ref{def-SL-Had-m},
applied with $k=1$.

Now let $2< r \le n+1$.
For $t>2$, $t\neq r$ we have 
\begin{eqnarray*}
\la v^{\{r,1\}},v^{\{r,t\}}\ra 
&=& 
\la h_{r-1} , h_{r-1}\circ h_{t-1} \ra
= \la h_{r-1}\circ h_1 , h_{r-1}\circ h_{t-1} \ra
= \\
&=&\la h_1, h_{t-1} \ra =0
\end{eqnarray*}
as well as
\[
\la v^{\{r,2\}},v^{\{r,t\}}\ra 
= 
\la h_{r-1}\circ h_{r-1} , h_{r-1}\circ h_{t-1} \ra
= \la h_{r-1},h_{t-1}\ra =0.
\]
Finally we have 
\begin{eqnarray*}
\la v^{\{r,1\}},v^{\{r,2\}}\ra 
&=& \la h_{r-1} , h_{r-1}\circ h_{r-1} \ra
\\
&=& \la h_1 \circ h_{r-1} , h_{r-1}\circ h_{r-1} \ra = 
\la h_1, h_{r-1} \ra = 0.
\end{eqnarray*}

We note that $|\B|=n+1$ is odd since $n$ is assumed to be even.
We will complete the proof by verifying that 
condition (3) in Definition \ref{def-KS-pair} is satisfied.
If the mapping $\{i,j\}\mapsto v^{\{i,j\}}$ is injective,
then each $v^{\{i,j\}}$ belongs to exactly two entries of $\B$,
namely $B_i$ and $B_j$. If the list $(v^{\{i,j\}})_{1\le i<j\le n+1}$ 
contains repeated vectors,
then let $x$ be a vector that occurs exactly $t$ times in this list.
Then by the previous argument $x$ belongs to exactly $2t$ entries of $\B$,
since for distinct $i,j,k$ we have $v^{\{i,j\}}\neq v^{\{i,k\}}$ as 
$\la v^{\{i,j\}}, v^{\{i,k\}}\ra =0$.
\end{proof}

There are infinitely many dimensions $n$ to which
Theorem \ref{thm-Had-KS} applies.
By considering Proposition \ref{proposition-GH-SH},
a sufficient condition for 
applying Theorem \ref{thm-Had-KS} in an even dimension~$n$
is the existence of GH$(g,n/g)$ over $\Z_g$ for some $g>2$.
The simplest forms of an
infinite family for which this condition is
satisfied are $n=2^kp^m$ where $k\in\{1,2\}$, $p$ is an odd prime
and $m\ge 1$, or $n=8p$ where $p>19$ is prime.
The constructions for the underlying generalized Hadamard
matrices can be found by consulting Table~5.10 in \cite{handbook-GH}.
Furthermore, any known generalized Hadamard
matrices can be used as ingredients to recursive constructions
given in Theorems 5.11 and 5.12 in  \cite{handbook-GH}.
As these recursive constructions can be applied repeatedly,
it is impossible to give a closed form for all $n$
to which Theorem \ref{thm-Had-KS} applies.
For $n\le 100$ such $n$ are listed in Example~\ref{example-GH-exist} above.

\section{Conclusion}
\label{sec-conclusion}

A Kochen-Specker pair $(\V,\B)$ in $\C^6$ with $|\V|=21$ and $|\B|=7$
was recently discovered \cite{Lis-PRA}.
It was noted \cite{T7-experiment} as the 
{\em simplest} Kochen-Specker pair (KS pair)
since it strictly minimizes the cardinality of $\B$ among
all known KS~pairs $(\V,\B)$,
see \cite{WA-preprint}. This KS pair was originally found
by computer search and its internal structure has not been 
fully studied yet. In this paper we have revealed the
structure of this KS~pair, since it is obtained 
by applying Theorem \ref{thm-Had-KS} 
in the smallest possible dimension $n=6$.
We have discovered an application of generalized Hadamard matrices
to the construction of KS pairs,
and we have proposed a new class of Hadamard matrices
as the suitable domain for such constructions.

\end{document}